\documentclass[11pt, a4paper]{amsproc}
\usepackage{amsfonts,amsmath,amssymb, amscd}
\usepackage{txfonts}

\newtheorem{theorem}{Theorem}[section]
\newtheorem{lem}[theorem]{Lemma}
\newtheorem{lemma}[theorem]{Lemma}
\newtheorem{definition}[theorem]{Definition}
\newtheorem{prop}[theorem]{Proposition}
\newtheorem{cor}[theorem]{Corollary}

\theoremstyle{definition}

\newtheorem{example}[theorem]{Example}
\newtheorem{notation}[theorem]{Notation}

\newcommand\pf{\begin{proof}}
\newcommand\epf{\end{proof}}

\def\CX{{\mathbb C}}
\def\QX{{\mathbb Q}}
\def\OQX{{\overline{\mathbb Q} }}
\def\NX{{\mathbb N}}

\def\ZX{{\mathbb Z}}
\def\RX{{\mathbb R}}
\def\PX{{\mathbb P}}

\def\GL{{\rm GL}}

\def\gl{{\rm gl}}

\def\d{{
\partial}}

\def\calS{{\mathcal S}}

\def\calF{{\mathcal  F}}

\def\calK{{\mathcal K}}

\def\calN{{\mathcal  N}}

\def\cf{{\em cf.\ }}
\def\calD{{\mathcal D}}
\def\calC{{\mathcal C}}
\def\calF{{\mathcal F}}
\def\calO{{\mathcal O}}

\def\calOU{{\calO_{\calU}}}
\def\calOUU{{\calO_{\calU'}}}

\def\calS{{\mathcal S}}
\def\calU{{\mathcal U}}
\def\calV{{\mathcal V}}

\def\d{{
\partial}}

\def\pp{{\mathbb{P}}^1({\mathbb{C}})}

\def\Ga{{{\mathbb G}_a}}
\def\Gm{{{\mathbb G}_m}}

\begin{document}

\title[Applications of parameterized Picard-Vessiot theory]{Some applications of the \\ parameterized Picard-Vessiot theory}

\author{Claude Mitschi}
\address{Institut de Recherche Math\'ematique Avanc\'ee, Universit\'e de Strasbourg, 7 rue Ren\'e Descartes, 67084 Strasbourg Cedex, France}
\curraddr{}
\email{mitschi@math.unistra.fr}


\date{}


\begin{abstract}
This expository article, intended for a special volume in memory of Andrey Bolibrukh, describes some applications of the parameterized Picard-Vessiot theory. This  Galois theory for parameterized linear differential equations was  Cassidy and Singer's contribution to  an  earlier volume dedicated to  Bolibrukh.  The main results we present here were  obtained in joint work  with Michael Singer, for families of ordinary differential equations with parameterized regular singularities. They include `parametric' versions of the Schlesinger theorem and of the weak Riemann-Hilbert problem as well as an algebraic characterization  of a special type of monodromy evolving deformations, illustrated by  the classical Darboux-Halphen equation.  Some of these results were recently applied by different authors to  solve the inverse problem in parameterized Picard-Vessiot theory, and  were also generalized to irregular singularities. We  sketch some of these results by other authors. The paper includes a brief history of the Darboux-Halphen equation as well as an appendix about differentially closed fields. 
\end{abstract}

\maketitle

\tableofcontents

\section{Parameterized Picard-Vessiot theory}
The classical Picard-Vessiot theory, or differential Galois theory, PV-theory for short, associates with any linear differential system 
\begin{eqnarray}\label{eqclassPV}
\d Y &= &A Y 
\end{eqnarray}
where the entries of the square matrix $A$  belong to a differential field $k$ of characteristic zero with derivation $\d$ and  algebraically closed field of constants,   a so-called Picard-Vessiot extension of $k$. This is a differential field extension of $k$   generated by the entries of a fundamental solution,  it has no new constants and its derivation is given by (\ref{eqclassPV}). Picard-Vessiot extensions are unique up to differential $k$-isomorphisms, and their group of differential $k$-automorphims is called the Picard-Vessiot group, or differential Galois group. It is a linear algebraic group,  which reflects many properties of the equation, such as its solvability, reducibility, existence of algebraic solutions etc.

In the special volume \cite{CaSi} dedicated to Andrey Bolibrukh, Cassidy and Singer developed a {\it parameterized Picard-Vessiot  theory}, PPV-theory for short, based on seminal work by Cassidy, Kolchin and Landesman. 
In PPV-theory, the differential base-field $k$ is endowed with a set of commuting derivations $\Delta=\{\d_0, \d_1\ldots \d_m\}$. As in  
PV-theory, one wants to associate with a (square) differential system
\begin{eqnarray}\label{eqPPV}
\d_0 Y &= &A Y 
\end{eqnarray}
with coefficients in $k$, a unique {\em parameterized Picard-Vessiot extension}, that is, a $\Delta$-differential field extension of $k$ generated by the entries of a  fundamental solution of (\ref{eqPPV}) (generated as a field extension by these entries and  their $\Delta$-derivatives at any order) with no new $\d_0$-constants.  The  {\em parameterized Picard-Vessiot group} of a PPV-extension is  its group of $\Delta$-differential $k$-automorphisms, with  the usual expected properties such as  a parameterized version of  ``Galois correspondence".   The following example (\cf   \cite{CaSi} p.118) shows that some asumptions are needed to meet these requirements.

\begin{example} \label{toymodel} Consider the scalar differential equation 
\begin{eqnarray}\label{eqntype}
\frac{dy}{dx}=\frac{t}{x} y
\end{eqnarray}
For fixed $t\in\CX$ we can apply classical PV-theory  over the differential fields $\CX(x)$ or $\OQX(x)$ for instance. An easy  calculation on the  solution $x^t$ shows that the PV-group of (\ref{eqntype}) over $\CX(x)$ (resp. $\OQX(x)$) is $\CX^*$ (resp. $\OQX^*$) if $t\notin\QX$,  a  cyclic subgroup (of roots of unit)  else.

If we now consider (\ref{eqntype}) as a  parameterized family 
over the differential field $k=\CX(x,t)$ of rational functions in $x$ and $t$ with  derivations $\{\frac{d}{dx},\frac{d}{dt}\}$, its PPV-extension is
\[ K=\CX(x,t,x^t,\log x).\] 
Let us show that the corresponding PPV-group is
\[ G=\CX^*\]
and that $\log x$ is an element of $K$ invariant by $G$, whereas  the subfield $K^G$ of $K$ of elements left invariant by $G$ should be the base-field $k$ if $G$ satisfied Galois correspondence. 
Since an element  $\sigma\in G$ is determined by $\sigma(x^t)$ and $\sigma(\log x)$ and  commutes with both derivations,  it is of  the form 
\[\sigma(x^t)=a_{\sigma} x^t, \quad \sigma(\log x)=\log x + c_\sigma \]
where  $c_{\sigma}$ is the logarithmic derivative of $a_{\sigma}$, and ~$a_{\sigma}\in \CX^*$,~$c_{\sigma}\in \CX$ only depend on~$t$. An easy calculation shows  that
 \[ G=\{a\in \CX(t)^*, a'' a -a'^2=0\}=\CX^* \]
where $a'$, $a''$ are the first and second derivatives with respect to $t$, and that $G=\CX^*$ since the $a$ are rational functions of $t$, which in particular implies that $c_{\sigma}=0$ for all $\sigma\in G$, hence $\sigma(\log x)=\log x$ for all $\sigma\in G$.    
\end{example}

\medskip
To have  $K^G=k$ in Example \ref{toymodel}, the group $G$ needs to be larger, hence contain non-constant elements.  If one assumes the field $k_0=k^{\d_0}$ of $\d_0$-constants to be {\it differentially closed} (see the Appendix) then Cassidy and Singer (\cite{CaSi}, p.116) proved that for any equation (\ref{eqPPV}) there is a unique PPV-extension of $k$ and that its PPV-group is a linear {\em differential algebraic group}  defined over $k_0$, that is, a subgroup of $\GL(n,k_0)$ defined by differential polynomial equations, in other words, closed in the Kolchin topology, whose elementary closed sets are the zero sets of $\{\d_1,\ldots,\d_m\}$-differential polynomials.  For  more facts about differential algebraic groups we refer to the work of Cassidy \cite{cassidy}, who first introduced these objects, and to \cite{kolchin_groups}, \cite{Buium}. A Galois correspondence now holds between closed  differential subgroups of the PPV-group and  intermediate $\Delta$-differential extensions of~$k$  in the PPV-extension.

Note that since $k_0$ is assumed to be differentially closed, it  is in particular algebraically closed, and  usual PV-theory   holds  for  Equation (\ref{eqPPV}). The PPV-group, which is Kolchin-closed in $\GL(n,k_0)$, is  not closed in general in the (weaker)  Zariski-topology and its Zariski-closure is precisely the PV-group.  

In what follows we  only consider  families of differential equations whose coefficients are complex analytic functions,  depending analytically on complex parameters. In the parametric case we first need  to clarify the notion of  regular singular points. 

\section{Parameterized singular points}

Consider a family of linear differential equations
\begin{eqnarray}\label{PPVsystem}
{\frac{\d Y}{\d x}}&=&A(x,t) Y
\end{eqnarray}
parameterized by $t$, where $A\in\gl_n\bigl(\calOU(\{x-\alpha(t)\})\bigr)$ depends analytically on $x$ and~$t$, as explained in  the  notation below.

In what follows we will use  the words `system' or `equation' indifferently for  a matricial equation, that is, a system of equations.

\begin{notation}
$\calU\subset\CX^r$ is an open connected subset containing  $0$,  $\calO_{\calU}$ is the ring of analytic functions on $\calU$ of the multi-variable $t$,  and $\alpha\in \calOU$, with $\alpha(0)=0$ can be thought of  as a moving singularity near $0$. Let  
$\calOU((x-\alpha(t)))$ denote the ring of formal Laurent series  with coefficients in $\calOU$ 
\[  f(x,t)=\sum_{i\ge m} a_i(t)(x-\alpha(t))^i\]
where $m\in \ZX$ is independent of $t$, 
 and let $\calOU(\{x-\alpha(t) \})$  denote the ring  of  those $f(x,t)\in\calOU((x-\alpha(t)))$  that,   for  each fixed  $t\in \calU$, converge  for $0<|x-\alpha(t)|<R_t$, for some $R_t>0$.

\smallskip
\noindent Note that in  a compact neighbourhood  $\calN\subset\calU$ of $0$, one can choose $R_t$ to be independent of $t$, for $t\in\calN$.

\end{notation}

With these asumptions and notation, we can expand the matrix $A$  in (\ref{PPVsystem}) as
\[
A(x,t)=\frac{A_{-m}(t)}{(x-\alpha(t))^{m}}+\frac{A_{-m+1}(t)}{(x-\alpha(t))^{m-1}}+\ldots=\sum_{i\ge{-m}} {(x-\alpha(t))^{i}}{A_i(t)}
\]
where $A_i(t)\in\gl_n\bigl(\calOU\bigr)$ for all $i\ge{-m}$, and $m\in\NX$ does not depend on~$t$.

\begin{definition}\label{equiv}

Two parametric equations
\[ {\frac{\d Y}{\d x}}=AY    \quad{\mathrm{and}}\quad    {\frac{\d Y}{\d x}}=BY, \]
with $A,B\in\gl_n\bigl(\calOU(\{x-\alpha(t)\})\bigr)$ are {\em equivalent} if for some $P\in \GL_n\bigl(\calOU(\{x-\alpha(t)\})\bigr)$ 
\[B={\frac{\d P}{\d x}}P^{-1}+PAP^{-1}. \]
 
\end{definition}

\begin{definition}\label{regsing} With notation as before,

\begin{enumerate}
\item Equation (\ref{PPVsystem}) has {\em simple singular points  near $0$} if $m=1$ and $A_{-1}\ne 0$ as an element of $\gl_n\bigl(\calOU\bigr)$,

\item Equation (\ref{PPVsystem}) has {\em parameterized regular singular points near $0$} (notation $prs_0$) if it is equivalent to an equation with simple singular points near $0$.
\end{enumerate}
\end{definition}

\begin{example} Let 
\begin{eqnarray*}
A & = &\left(\begin{array}{cc} 0 & -3\\ 0& 0 \end{array}\right) \frac{1}{ (x-t)^2} + \left(\begin{array}{cc} t & 0\\ 0& t-2 \end{array}\right) \frac{1}{ x-t}\\
B & = & \left(\begin{array}{cc} t-1&0\\ 0& t-1 \end{array}\right) \frac{1}{x-t}
\end{eqnarray*}

These equations are equivalent  via  
\begin{eqnarray*}
P & = & \left(\begin{array}{cc} \frac{1}{x-t}&\frac{-1}{(x-t)^2}\\ 0& x-t \end{array}\right) 
\end{eqnarray*}
and since the latter  has  simple singular points near $0$, the first equation  has parameterized regular singular points   near $0$.
\end{example}

In analogy to the non-parameterized case, solutions of an equation (\ref{PPVsystem}) with parameterized regular singularities near $0$ have  ``uniformly"  a  moderate growth as $x$ gets near $\alpha(t)$ and $t$ tends to $0$ (\cf \cite{MS0}, Cor. 2.6).

\begin{prop}\label{modgrowth} Assume that Equation (\ref{PPVsystem}) has regular singular points near $0$. Then there is an open  connected subset $\calU'$  of $\calU$ such that
\begin{itemize}
\item[1)]  Equation (\ref{PPVsystem}) has a solution $Y$ of the form
\begin{eqnarray}\label{sol}
Y(x,t)&=&\Bigl( \sum_{i\ge i_0}(x-\alpha(t))^i Q_i(t) \Bigr) (x-\alpha(t))^{{\tilde{A}}(t)}
\end{eqnarray}
with  ${\tilde{A}}\in\gl_n(\calOUU)$ and $Q_i\in\gl_n(\calOUU)$ for all $i\ge i_0$,

\smallskip
\item[2)] for any $r$-tuple $(m_1,\ldots,m_r)$ of non-negative integers there is an integer $N$ such that for any fixed $t\in{\mathcal U}'$ and  any sector ${\mathcal S}_t$ from $\alpha(t)$ in the complex plane, of opening less than $2\pi$, 
\[Ê\lim_{x\rightarrow\alpha(t)\atop x\in {\mathcal S}_t} \bigl(x-\alpha(t)\bigr)^N \frac{\d ^{m_1+\ldots+m_r}Y(x,t)}{\d^{m_1}t_1\ldots\d^{m_r}t_r}=0.\] 

  \end{itemize}

\end{prop}
Solutions of parameterized differential equations with {\em irregular} singularities have been studied by Babbitt and Varadarajan  in  \cite{babbitt_varadarajan},  by Sch\" afke in  \cite{schaefke}, and more recently by Dreyfus in \cite{TD}. Assuming $0$ is a (non-moving)  irregular singularity, these authors gave a  condition on the exponential part of a formal solution in its usual form
\[ {\hat{Y}}(z)={\hat{H}}(z)z^Je^Q\]
ensuring that the coefficients of the formal series ${\hat{H}}(z)$ depend analytically on the multi-parameter.

\section{PPV-theory  and monodromy}
 From the beginning of Picard-Vessiot theory in the nineteenth century, monodromy has been closely related to the `group of transformations' of linear differential equations, now called the Picard-Vessiot group. More information about the history of  the monodromy group and the Picard-Vessiot group can be found  in \cite{Borel2} and \cite{Zola}. 

\subsection{Classical Picard-Vessiot theory and monodromy}

In classical PV-theory it is commonly admitted that the ``monodromy matrices belong to the differential Galois group", which is in particular true for a differential equation (\ref{eqclassPV}) over the base-field $\CX(x)$, but which does not hold over ${\overline{\QX}}(x)$ though. Moreover, if (\ref{eqclassPV})  has regular singular points only, Schlesinger's theorem (\cf \cite{schlesinger}, \S\ 159,160, \cite{PuSi2003} Th.5.8) tells us that the monodromy matrices generate a Zariski-dense subgroup of the differential Galois group over $\CX(x)$. For instance,  in Example \ref{toymodel} above: 
\[
\frac{dy}{dx}=\frac{t}{x} y
\]
let $t$ denote a constant non-zero complex number. This equation has two regular singular points, at $0$ and $\infty$. With respect to the solution $x^t$ (for a given determination of $\log x$) the monodromy `matrices' with respect to $0$ and $\infty$ are the scalars $m_0=e^{2\pi i t}$ and $m_{\infty}=e^{-2\pi i t}$.   It is easy to see that the Zariski closure in $\CX^*$ of the subgroup generated by $m_0$ (or $m_{\infty}$) is the PV-group over $\CX$ given above ($\CX^*$ or a finite cyclic group). 
  
If   $t\in\OQX$,  what happens over the differential field $\OQX(x)$ ? The monodromy scalars $e^{\pm 2\pi i t}$ may be transcendental in this case and hence not belong to the PV-group, which  is  a subgroup of  $\OQX^*$.  But the results given earlier show that the PV-group is defined by the same equation in $\CX^*$ or $\OQX^*$ respectively, whether we consider $t\in\CX^*$ or  $t\in \OQX^*$.  The following example too illustrates the importance of the base field. 

\begin{example}\label{pbmon}
\begin{eqnarray*}
\frac{dY}{dx} & = &\left(\begin{array}{cc} 1/x & 1\\ 0& 0 \end{array}\right) Y.\end{eqnarray*}
This equation has two regular singular points, one Fuchsian  at $0$, one at $\infty$. With respect to the fundamental solution
\begin{eqnarray*}
\left(\begin{array}{cc} x & x\log x\\ 0& 1 \end{array}\right) 
\end{eqnarray*}
the monodromy matrix at $0$ is
\begin{eqnarray*}
M&=&\left(\begin{array}{cc} 1 & 2\pi i\\ 0& 1 \end{array}\right). 
\end{eqnarray*} 
If we consider the equation over $\OQX(x)$, clearly  $M$ does not belong to the PV-group over $\OQX(x)$ since it has a transcendental entry.
\end{example}

To adjust Schlesinger's result to this situation we use the following result (\cf \cite{MS0}, Prop. 3.1 and Cor. 3.2)

\begin{prop}\label{letprop1} 
Let $C_0\subset C_1$ be algebraically closed fields and $k_0 = C_0(x), k_1 = C_1(x)$ be differential fields where $c' = 0$ for all $c \in C_1$ and $x'=1$.
Let 
\begin{eqnarray}\label{leteqn0}Y' &=& AY\end{eqnarray}
be a differential equation with $A
 \in \gl_n(k_0)$. If $G(C_0)\subset\GL_n(C_0)$ is the PV-group over $k_0$ of Equation (\ref{leteqn0}) with respect to some fundamental solution,  where $G$ is a linear algebraic group defined over $C_0,$  then $G(C_1)$ is the PV-group of (\ref{leteqn0}) over $k_1$,  with respect to some fundamental solution. 
\end{prop}
For instance, on  Example \ref{pbmon}, we  easily see that the PV-group over $\OQX(x)$ is 
\begin{eqnarray*}
G=&\left\{\left(\begin{array}{cc} 1 & \lambda\\ 0& 1 \end{array}\right), \lambda\in\OQX \right\} 
\end{eqnarray*} 
and the PV-group over $\CX(x)$ is the  group of $\CX$-points of $G$
\begin{eqnarray*}
G(\CX)=&\left\{\left(\begin{array}{cc} 1 & \lambda\\ 0& 1 \end{array}\right), \lambda\in\CX \right\}. 
\end{eqnarray*}.

The monodromy matrices do belong to the PV-group, after extending scalars. 
\begin{cor}\label{cstmonPV}
 Assume  in Equation (\ref{leteqn0}) that $A \in \gl_n(C_0(x))$ where $C_0$ is some algebraically closed subfield of $\CX$. Assuming  $0$ is a non-singular point, let us  fix it as the base-point of $\pi_1(\PX^1(\CX)\backslash \calS)$, where  $\calS$ is the set of singular points of (\ref{leteqn0}) on  $\PX^1(\CX)$. Let $G(C_0)$ be the PV-group of (\ref{leteqn0}) over $C_0(x)$, where $G$ is a linear algebraic group defined over $C_0$. If $C_1$  is any algebraically closed subfield of $\CX$  containing $C_0$ and the entries of the monodromy matrices, then the monodromy matrices are elements of the PV-group $G(C_1)$ of  (\ref{leteqn0}) over $C_1(x)$.
\end{cor}

\medskip
\subsection{Monodromy matrices in the PPV-group}
In PPV-theory too,  the equation may have coefficients  in some differentially closed field and the entries of the parameterized monodromy matrices  not belong to this field.

In \cite{MS0} we proved  a result similar to Proposition~\ref{letprop1} for parameterized Picard-Vessiot extensions. 
Consider equations of the form
\begin{eqnarray} \label{leteqn1} \d_x Y &=& A(x,t)Y\end{eqnarray}
where $A(x,t)\in\gl_n(\calOU(x))$ and $t=(t_1,\ldots,t_r)\in \calU$ for  some domain $\calU\subset\CX^r$. Denoting  differentiation with respect to $x, t_1,\ldots,t_r$ by $\d_x, \d_{t_1},\ldots,\d_{t_r}$ respectively,  let  $\Delta=\{\d_x, \d_{t_1},\ldots,\d_{t_r}\}$  and $\Delta_t
=\{\d_{t_1},\ldots,\d_{t_r}\}$. 

Let $C$ be a $\Delta_t$-differentially closed  extension
 of some field of functions that are  analytic  on some domain of~$\CX^r$ and let $\d_{t_i}$ denote for each $i$  the derivation extending $\d_{t_i}$. We consider the  $\Delta$-differential field structure on $k =C(x)$ given by $\d_x(x) = 1, \d_{t_i}(x) = 0$ for each $i$ and $\d_x(c) = 0$ for all $c\in C$, and we  assume  that $A\in gl_n(k)$.

\begin{prop}\label{letprop2}Let $C_0\subset C_1$ be differentially closed $\Delta_t$-fields as $C$ above, inducing a $\Delta$-field structure on $k_0 = C_0(x)$ and  $k_1 = C_1(x)$. Let 
\begin{eqnarray}\label{leteqn2}\d_x Y&=& AY\end{eqnarray}
be a differential equation with $A \in \gl_n(k_0)$.  If $G(C_0)\subset\GL_n(C_0)$ is the PPV-group over $k_0$ of Equation (\ref{leteqn2}) with respect to some fundamental solution,  where $G$ is a linear differential algebraic group defined over the differential $\Delta_t$-field $C_0,$  then $G(C_1)$ is the PPV-group over $k_1$ of (\ref{leteqn2}) with respect to some fundamental solution. 
\end{prop}

Let us define the parameterized monodromy matrices, which belong to the PPV-group in the same sense as in the non-parameterized case, after extending the base-field.

 Let $\calD$ be an open subset of $\pp$ with $0 \in \calD$.   Assume that $\pp\backslash \calD$ is the union of $m$ disjoint disks $D_i$ and that for each $t \in \calU$, Equation (\ref{leteqn1})  has a unique singular point in  $D_i$.  Let $\gamma_i$,  $i = 1, \ldots , m$ be the elementary  loops  generating $\pi_1(\calD, 0)$. Let us fix a fundamental solution~$Z_0$ of (\ref{leteqn1}) in the neighborhood of $0$ and define, for each fixed $t\in\calU$, the monodromy matrices of (\ref{leteqn1}) with respect to this solution and the $\gamma_i$. These matrices, which depend on $t$, are by definition the {\it parameterized monodromy matrices} of Equation  (\ref{leteqn1}). 
 
 To prove that the monodromy matrices belong to the PPV-group we need, as in the non-parameterized case, to perform `analytic continuation'  of a polynomial expression $P(Z_0)$ in the entries of $Z_0$, where $P$ is a polynomial with coefficients in $C_0(x)$, over some differentially closed field $C_0$ not contained in $\CX$. The following  result of Seidenberg \cite{sei58,sei69} gives these coefficients, and hence $P(Z_0)$, an existence as analytic functions.
 \begin{theorem}[Seidenberg]\label{seidenberg} Let $\QX\subset \calK \subset \calK_1$  be finitely generated differential extensions of the field of rational numbers $\QX$, 
and assume that  $\calK$   consists of
meromorphic functions on some domain $\Omega\in\CX^r$. Then  $\calK_1$ is isomorphic to a field ${\calF}$ of functions that are meromorphic on a domain $\Omega_1 \subset \Omega$, such that $\calK|_{\Omega_1} \subset \calF$.
\end{theorem}
 
This leads to the expected analogue of Corollary \ref{cstmonPV}:
\begin{theorem}\label{mongal}
Assume in Equation (\ref{leteqn1}) that $A\in\gl_n(C_0(x))$, where   $C_0$ is any differentially closed $\Delta_t$-field containing  $\CX$ and let  $C_1$ be  any differentially closed $\Delta_t$-field containing $C_0$ and the entries of the parameterized monodromy matrices  of Equation (\ref{leteqn1}) with respect to a fundamental solution of (\ref{leteqn1}). Then the parameterized  monodromy matrices belong to   $G(C_1)$, where $G$ is the PPV-group of (\ref{leteqn1}) over the  $\Delta$-field $C_0(x)$. 
\end{theorem}

\subsection{A parameterized version of Schlesinger's theorem}\label{pvst}
Consider a family of  equations
\begin{eqnarray}\label{leteqn3}
{\frac{\d Y}{\d x}}&=&A(x,t) Y
\end{eqnarray}
where the entries of $A$ are rational in $x$, and analytic in $t$ in some open subset $\calU$ of $\CX^r$.
 Let as before $\calD$ be an open subset of $\pp$ with $0 \in \calD$.   Assume that $\pp\backslash \calD$ is the union of $m$ disjoint disks $D_i$ and that for each $t \in \calU$, Equation (\ref{leteqn3})  
has a unique singular point $\alpha_i(t)$ in  each $D_i$, and no singular points otherwise.  Let $\gamma_i$,  $i = 1, \ldots , m$ be the elementary  loops  generating $\pi_1(\calD, 0)$. Locally at $0$ we can fix a fundamental solution $Z_0$,  analytic in $\calV\times\calU$ where $\calV$ is neighbourhood of $0$ in $\calD\subset\CX$ and $\calU$ a neighbourhood of $0$ in $\CX^r$. 
Let as before 
$\Delta=\{\d_x, \d_{t_1},\ldots,\d_{t_r}\}$  and 
 $\Delta_t
=\{\d_{t_1},\ldots,\d_{t_r}\}$. 

\smallskip
 In \cite{MS0} we proved the following parameterized analogue of Schlesinger's theorem. 
 \begin{theorem}\label{schlesinger}
With notation and asumptions as before, assume that Equation (\ref{leteqn3}) has parameterized regular singularities only, near each $\alpha_i(0)$, $i=1,\ldots,m$. Let $k$ be a  differentially closed $\Delta_t$-field containing  the $x$-coefficients of the entries of $A$, the singularities $\alpha_i(t)$ of (\ref{leteqn3}) and the entries of the parameterized monodromy matrices with respect to $Z_0$. Then the parameterized monodromy matrices  generate a Kolchin-dense subgroup of $G(k)$, where $G$ is the PPV-group of (\ref{leteqn3})  over $k(x)$.
\end{theorem}

\begin{proof} To prove this theorem it is sufficient, by the Galois correspondance of PPV-theory,   to show that any element of the PPV-extension $k(x)\langle Z_0\rangle$ ($\Delta$-differentially generated by a fundamental solution $Z_0$) that is left invariant by the action of the parameterized monodromy matrices, is an element of the base-field $k(x)$. Fix such an $f\in k(x)\langle Z_0\rangle$, invariant by all the parameterized monodromy matrices. The idea of the proof is the following. Let $\calF_0$ be  the differential $\Delta_t$-subfield of $k$ generated over $\QX$ by the $x$-coefficiens of $A$, the singular points $\alpha_i(t)$  and the entries of the parameterized monodromy matrices (with respect to the elementary loops around the $\alpha_i(t)$). Let further $\calF_1$ denote any $\Delta_t$-subfield of $k$ containing $\calF_0$ such that $f\in\calF_1(x)\langle Z_0\rangle$. By Seidenberg's theorem \ref{seidenberg}, we can see $f$ as a meromorphic  function  on a suitable domain of the $(x,t)$-space. Since for each fixed $t$, the function $f$ is invariant by the monodromy matrices and has moreover  moderate growth at each singular point by Prop. \ref{modgrowth}, it is indeed a rational function of $x$. Note that, as in the non-parameterized case, since $f$ is single-valued, it has an isolated pole at each singular point of the equation (\cf \cite{Kuga}, Preparation Theorem 18.2 p.118). To show that it is globally a rational function of $x$, we apply the  lemma below, inspired by a result of R. Palais \cite{Pa78}.
 \end{proof}
 
 \begin{lem}\label{palais} Let $\calF$ be a $\Delta$-field of functions that are meromorphic on $\calV \times \calU$ where $\calV\subset\CX$ and $\calU\subset\CX^r$ are open connected sets, and let $\calC_x=\{u\in\calF\  |\  \d_xu=0 \}$. Furthermore assume $x \in \calF$. Let $f\in\calF$ be such that $f(x,t)\in\CX(x)$ for each $t\in\calU$. Then for some $m\in\NX$, there exist $a_0,\ldots,a_m,b_0,\ldots, b_m\in\calC_x$ such that
\[ f(x,t)=\frac{\sum_{i=0}^m a_ix^i}{\sum_{i=0}^m b_ix^i}\]
\end{lem}

\section{PPV-characterization of isomonodromy} 
 Let us first recall  that classical differential Galois theory, or PV-theory, extends easily and naturally to differential fields with several derivations. More precisely, let $k$ be a $\Delta$-differential field with derivations $\Delta=\{\d_0,\d_1\ldots,\d_r\}$, and consider a linear system of  equations
 \begin{eqnarray}\label{fullsyst}
\left\{ \begin{array}{ccc}
\d_0 Y &= &A_0Y\\
\d_1 Y&=&A_1Y \\
 &\vdots& \\
\d_r    Y   &=&A_rY\\
\end{array}
\right. 
\end{eqnarray}
where $A_0, A_1,\ldots,A_r\in\gl(n,k)$.

Assuming the subfield of $\Delta$-constants $C$ of $k$ is algebraically closed, for each  system (\ref{fullsyst}) there is a unique PV-extension $K$ of $k$, that is, a $\Delta$-differential extension of $k$ generated by the entries of a fundamental solution of (\ref{fullsyst}) with no new $\Delta$-constants. The corresponding $PV$-group of differential $k$-automorphisms of  $K$ is a linear algebraic group  $G\subset \GL(n,C)$, unique up to differential isomorphism, and satisfying Galois correspondence.

\subsection{Integrable systems}
 The notion of integrability has a nice interpretation in terms of PPV-theory.  Integrability, over abstract differential fields, has the same definition as over fields of analytic functions (\cf \cite{CaSi}).
\begin{definition} With notation as above
\begin{enumerate}

\item the differential system (\ref{fullsyst}) is {\em integrable} if
\[ \d_i A_j-\d_jA_i=[A_i,A_j] \]
for all $0\le i,j\le r$, where $[\ ,]$ denotes the Lie bracket,

\medskip
\item an equation \[ \d_0 Y=AY, \quad A\in\gl(n,k)\]
is {\em completely integrable} if it can be completed into a system (\ref{fullsyst})  with $A_0=A$.
\end{enumerate}

\end{definition}

For completely integrable equations, PV-theory and PPV-theory get close (\cf \cite{CaSi}), Lemma 9.9).
 \begin{lemma} With notation as above,  assume the field $k_0$ of $\d_0$-constants of $k$ is $\Delta$-differentially closed, and let 
\begin{eqnarray}\label{eqsimple}
\d_0 Y=A Y, \quad A\in\gl(n,k)
\end{eqnarray}
be a completely integrable system,  completable into an integrable system (\ref{fullsyst}) as above.
Then any PV-extension of $k$ for (\ref{fullsyst}) is a PPV-extension of $k$ for (\ref{eqsimple}).
\end{lemma} 

The proof of this lemma relies on the fact that a differentially closed field is {\em a fortiori} algebraically closed, and that the field of constants of an algebraically closed differentially field is itself algebraically closed. This lemma was used by Cassidy and Singer  to give the following PPV-characterization of integrability (\cf \cite{CaSi},  Prop. 3.9).
\begin{prop} [Cassidy-Singer]\label{cassidysinger}With notation as above, assume $k_0$ is differentially closed, and let $C\subset k_0$ denote the subfield of $\Delta$-constants of $k$.
\begin{enumerate}
\item Equation (\ref{eqsimple}) is completely integrable if and only if its PPV-group over $k$ is conjugate in $\GL(n,k_0)$ to the group $G(C)$ of $C$-points of some  linear algebraic group defined over $C$.  
\item In particular, (1) holds if  $A\in\gl(C)$.
\end{enumerate}
\end{prop}

\subsection{Isomonodromy}

Let us again consider  the case of differential fields containing analytic functions. 
We consider as in Section \ref{pvst} a parameterized system  
\begin{eqnarray} \label{isomeq}
\d_x Y &=& A(x,t)Y\end{eqnarray}
where the entries of $A$ are  analytic on  $\calD\times \calU$ for some open subset $\calU\subset\CX^r$ containing $0$ and   some open subset $\calD$ of $\pp$ containing $0$ and such that $\pi_1(\calD,0)$ is generated by elementary loops $\gamma_1,\ldots,\gamma_m$. More precisely we assume that  $\pp\backslash \calD$ is the union of $m$ disjoint disks $D_i$ and that for each $t \in \calU$, Equation (\ref{isomeq})  
has a unique singular point $\alpha_i(t)$ in  each $D_i$, and no singular points otherwise.

\begin{definition} Equation (\ref{isomeq}) is {\em isomonodromic}  on $\calD\times \calU$ if there are constant matrices $M_1,\ldots,M_m\in\GL(n,\CX)$ such that for each fixed $t\in\calU$ there is a local fundamental solution $Y_t$ of (\ref{isomeq}) at $0$ such that analytic continuation $Y_t^{\gamma_i}$ of $Y_t$ along $\gamma_i$ yields
\[ Y_t^{\gamma_i}=Y_t M_i\]
for  $i=1,\ldots,m$.
\end{definition}

Note that $Y_t$ may  {\em a priori} not be analytic in $t$. Nevertheless, following a  proof by Andrey Bolibrukh in the Fuchsian case (\cf \cite{Bol_iso_def}), one can show the existence of such a solution $Y_t$ which is analytic in $t$, using in particular the fact that $\calU$ is a Stein variety, on which  any topological trivial (analytic) bundle is analytically trivial (\cf~\cite{Cartan}). 

\smallskip
A useful criterion for isomonodromy is the following. 
\begin{theorem}[Sibuya \cite{Sibuya}]
Consider an  equation (\ref{isomeq}) 
with notation and asumptions as above. 
\begin{enumerate}
\item Equation (\ref{isomeq}) is isomonodromic on $\calD\times \calU$ if and only it is completely integrable, that is, part of an integrable system

\begin{eqnarray*}\left\{ \begin{array}{ccc}
\d_0 Y &= &A_0Y\\
\d_1 Y&=&A_1Y \\
 &\vdots& \\
\d_r    Y   &=&A_rY\\
\end{array}
\right. 
\end{eqnarray*}
with $A_0=A$ and analytic  $A_i$  on $\calD\times \calU$ for all $i$. 

\smallskip
\item Assume (\ref{isomeq}) is isomonodromic. If moreover $A$ is rational in $x$ and Equation~(\ref{isomeq}) has parameterized regular singular points only, then  the entries of all  $A_i$ are rational in~$x$.
\end{enumerate}
\end{theorem}

 In \cite{CaSi} Cassidy and Singer give an algebraic criterion for isomonodromy using  PPV-theory. Let as before 
$\Delta=\{\d_x, \d_{t_1},\ldots,\d_{t_r}\}$ and  $\Delta_t
=\{\d_{t_1},\ldots,\d_{t_r}\}$  denote the partial differentiation with repect to $x$ and the multi-parameter $t$.
 
\begin{theorem}[Cassidy-Singer]
Consider an equation 
 \[\d_x Y = A(x,t)Y\] 
 as before, 
where $A$ has entries  analytic in $\calD\times\calU$, rational in $x$,   with parameterized regular singularities only, one in each disk $D_i$.  Let $k=C_0(x)$, where $C_0$ is a $\Delta_t$-differential closure of the field generated over $\CX(t_1,\ldots,t_r)$ by the $x$-coefficients of the entries (which are rational functions of $x$) of $A$ . This equation is isomonodromic if and only if its PPV-group over $k$ is conjugate in $\GL(n,C_0)$ to a linear algebraic subgroup of $\GL(n,\CX)$.
\end{theorem}

The proof of this theorem relies on Sibuya's criterion and  Proposition \ref{cassidysinger}.

\section{Projective isomonodromy} 
Consider as before a parameterized equation  
\begin{eqnarray} \label{projis}
\d_x Y = A(x,t)Y
\end{eqnarray}
 on $\calD\times \calU$ with $m$ isolated singular points, each in a disk $D_i$ such that $\calD=\pp\setminus \cup_{i=1}^m D_i$. We are now considering a special case of so-called {\em monodromy evolving deformations}, which has been  studied on the classical example of the Darboux-Halphen equation by   Chakravarty and Ablowits \cite{ChAb} and Ohyama (\cite{ohyama}, \cite{ohyama2}). 
 
 \begin{definition} Equation (\ref{projis}) is {\em projectively isomonodromic} if there are constant matrices $\Gamma_1,\ldots \Gamma_m \in \GL(n,\CX)$ and analytic functions $c_1,\ldots,c_m \in \calOU$ such that for each fixed $t\in \calU$ there is locally at $0$ a fundamental solution $Y_t$ of (\ref{projis}) such that  for each $i$ the parameterized monodromy matrix of (\ref{projis}) with respect to $Y_t$ and the loop $\gamma_i$ is \[c_i(t)\Gamma_i.\]
 \end{definition} 
 As in the isomonodromic case, the solution $Y_t$ may not be analytic in $t$ and in \cite{MS} we  mimick Bolibrukh's proof to show the existence of  such a particular solution that is analytic in $t$. We need  such a solution to interpret projective isomonodromy algebraically in terms of PPV-theory. 
 
 \smallskip
 In the special  case  of a Fuchsian parameterized equation 
 
 \begin{eqnarray}\label{Fparam}
\d_x Y&=&\sum_ {i=1}^m \frac{A_i(t)}{x-\alpha_i(t)} Y
\end{eqnarray}
projective isomonodromy  is related to isomonodromy in a natural  way (\cf \cite{MS}).

 \begin{prop}\label{Fprop}
Equation (\ref{Fparam}) is projectively isomonodromic if and only if for each i
\[ A_i(t)=B_i(t)+b_i(t) I  \]
where $b_i$ and the entries of $B_i$ are analytic on $\calU$ and such that the equation
\[ \d_x Y=\sum_ {i=1}^m \frac{B_i(t)}{x-\alpha_i(t)} Y\]
is isomonodromic.
\end{prop}

For general equations (\ref{projis}) with parameterized regular singularities we give in \cite{MS} an algebraic characterization of projective isomonodromy in terms of their PPV-group.
\begin{theorem}\label{projisomPPV}
With notation as before, consider a parameterized equation 
\begin{eqnarray}\label{projis2}
\d_x Y = A(x,t)Y
\end{eqnarray}
 where $A$ has entries  analytic in $\calD\times\calU$, rational in $x$, and assume that this equation has parameterized regular singularities only, one in each disk $D_i$.  Let $k=k_0(x)$, where $k_0$ is a $\Delta_t$-differential closure of the field generated over $\CX(t_1,\ldots,t_r)$ by the $x$-coefficients of the rational functions entries of $A$ . Then this equation is projectively isomonodromic if and only if its PPV-group over $k$ is conjugate in $\GL(n,k_0)$ to a subgroup of 
\[\GL(n,\CX) \cdot k_0 I\subset\GL(n,k_0)\]
where $k_0 I$ is  the subgroup  of scalar matrices of $\GL(n,k_0)$.
\end{theorem}

Combining  topological arguments in both the Kolchin and the Zariski topology, and using Schur's lemma we get a corollary of this result for {\em absolutely irreducible} equations over $k$, that is, equations that are irreducible over any finite extension of~$k$. We recall that an equation is said to be irreducible if  the corresponding differential polynomial is irreducible (it has no factor of strictly less order), equivalently if its differential  Galois group acts irreducibly on its  solution space in any Picard-Vessiot extension. 
\begin{cor}
Let $A,k_0$ and $k$ be as in  Theorem \ref{projisomPPV}. If  Equation (\ref{projis2}) is absolutely irreducible, then it is projectively isomonodromic if and only if the commutator subgroup $(G,G)$ of its PPV-group $G$ is conjugate in  $\GL(n,k_0)$ to a subgroup of $\GL(n,\CX)$.
\end{cor}

\section{The Darboux-Halphen equation} The  results of the previous section are well illustrated on the Darboux-Halphen equation. This equation describes projective isomonodromy in the same way as the Schlesinger equation accounts for  isomonodromy (of the Schlesinger type) for parameterized Fuchsian systems. 
The  Darboux-Halphen V equation
 
\begin{eqnarray*}
{\mathrm{ (DH\  V)}}\ \ \left\{ \begin{array}{cccccccc}
\omega_1' &= &&\omega_2\omega_3 &-&\omega_1(\omega_2+\omega_3)&+& \phi^2 \\
\omega_2' &=&& \omega_3\omega_1 &-&\omega_2(\omega_3+\omega_1)&+ &\theta^2 \\
\omega_3' &=&& \omega_1\omega_2 &-&\omega_3(\omega_1+\omega_2)&- &\theta\phi \\
\phi'         &=&&\omega_1(\theta - \phi) &-&\omega_3(\theta + \phi)&&\\
\theta'      &=& - &\omega_2(\theta- \phi) &-&\omega_3(\theta + \phi),&&\\
\end{array}
\right. 
\end{eqnarray*}
occurs in  physics as a reduction of the selfdual Yang-Mills equation (SDYM).
For~$\theta=\phi$,  (DH V) is equivalent to Einstein's selfdual vacuum equations.
For~$\theta=\phi=0$, it is  Halphen's original equation (H II), solving a  geometry problem of Darboux about orthogonal surfaces.

Contrary to other SDYM reductions such as the Painlev\'e equations,    (DH V) does not satisfy the Painlev\'e property, since it has a boundary of movable essential singularities. It is therefore  
not likely to rule isomonodromy.

  \subsection{History of the DH-equation.}
Halphen's equation (H II) goes back to Darboux's work (\cite{GD1}, \cite{GD2}) on orthogonal systems of surfaces.
Darboux's original problem  was the following.

\smallskip
\noindent{\em Problem 1}: What condition on a given pair $(\calF_1,\calF_2)$ of orthogonal families of surfaces  in $\RX^3$ implies that there exists a family $\calF_3$ such that $(\calF_1,\calF_2,\calF_3)$ is a   triorthogonal system  of  pairwise orthogonal families?

\smallskip
In \cite{GD1} Darboux gives a necessary and sufficient condition on  $(\calF_1,\calF_2)$ to solve the problem:  that the intersection of any  surfaces $S_1\in\calF_1$ and $S_2\in\calF_2$ be a curvature line of both $\calF_1$ and $\calF_2$. The necessary condition  was already known as Dupin's theorem (1813). 
 
 \smallskip
\noindent{\em Problem 2}: What condition on its parameter $u = \varphi(x,y,z)$ implies that a one-parameter family $\calF$ of surfaces in $\RX^3$ belongs to a {\em triorthogonal} system $(\calF_1,\calF_2,\calF_3)$, of three pairwise orthogonal families?

\smallskip
In \cite{GD2} Darboux found and solved an  order three partial differential equation satisfied by  $u$  and   obtained, based on previous work by Bonnet and Cayley,  the general solution from a particular family of ruled helicoidal surfaces. 
\'Elie Cartan~\cite{ElieCartan} later used  his  exterior differential calculus to prove that Problem 2 has a solution. He also generalized the
problem, replacing orthogonality by any prescribed angle, or considering $p$ pairwise orthogonal families of hypersurfaces in $p$-space.

\smallskip
Darboux  stated yet another problem on orthogonal surfaces. 

\smallskip
\noindent{\em Problem 3}: given two families $\calF_1$ and $\calF_2$ consisting each of parallel surfaces does there exist a family $\calF$ orthogonal to both $\calF_1$ and $\calF_2$ ?

\smallskip
It is an easy exercise  to prove that a solution  should either consist of planes, or of ruled quadrics.
If $\calF$ consists of quadrics {\em with a center}, these have simultaneously reduced equations:
\[\frac{x^2}{a(u)} + \frac{y^2}{b(u)}+\frac{z^2}{c(u)}=1\]
which depend on the parameter $u=\varphi(x,y,z)$ of $\calF$. One can show that $\calF$ solves Problem 3 if and only if $a,b,c$ satisfy the {\em Darboux equation}

\begin{eqnarray*}\label{D}
a(b'+c')=b(c'+a')=c(a'+b')
\end{eqnarray*}
where $a',b',c'$ are the derivatives with respect to $u$.
Darboux  could not solve the problem though:

\noindent{`\em These equations do not seem to be integrable by known procedures' }(Darboux,1878).

He  gave up on this part of the problem and restricted his study to centerless quadrics. He solved the particular problem with  a
family $\calF$  of paraboloids
\[ \frac{y^2}{\alpha+u}+\frac{z^2}{\alpha -u }=2x+\alpha \log u\] 
and  claimed that  some surfaces of revolution solved the problem as well.

\smallskip
In 1881 Halphen (\cite{H1}, \cite{H2}) completely solved Darboux's second problem in the following form:

\begin{eqnarray*}
{\mathrm{ (H\  I)}}\ \ \left\{ \begin{array}{ccccc}
\omega_1' &+&\omega_2'&=&\omega_1\omega_2  \\
\omega_2' &+&\omega_3'&=&\omega_2\omega_3  \\
\omega_3' &+&\omega_1'&=&\omega_3\omega_1  \\
\end{array}
\right. 
\end{eqnarray*}
known as the {\em Halphen I} equation, 
and  actually solved the more general QHDS (quadratic homogeneous differential system)
\begin{eqnarray*}
{\mathrm{ (H\  II)}}\ \ \left\{ \begin{array}{cccc}
\omega_1' &= &a_1\omega_1^2+(\lambda-a_1)(\omega_1\omega_2+\omega_3\omega_1- \omega_2\omega_3) \\
\omega_2' &= &a_2\omega_2^2+(\lambda-a_2)(\omega_2\omega_3+\omega_1\omega_2- \omega_3\omega_1) \\
\omega_3' &= &a_3\omega_3^2+(\lambda-a_3)(\omega_3\omega_1+\omega_2\omega_3- \omega_1\omega_2) \\
\end{array}
\right. 
\end{eqnarray*}
known as the {\em Halphen II equation},   by means of  hypergeometric functions.
He  considered  even more general QHDSs   
 \[ \{ \omega_r'=\psi_r(\omega_1,\ldots,\omega_l) \}_{ r=1,\ldots,l}\] 
where the
$\psi_r$ are quadratic forms, with some extra symmetry condition. A special example of such QHDS is  Equation ~(DH V)  above, and its particular form~(H II) which we consider now. 

\subsection{Application of PPV-theory to the Darboux-Halphen}
\smallskip
As shown in \cite{ohyama}, Equation  (H II) is equivalent to a system
\[x_i'=Q_i(x_1,x_2,x_3), \ \  i=1,2,3, \]
where $Q_i(x_1,x_2,x_3)=x_i^2+a(x_1-x_2)^2+b(x_2-x_3)^2+c(x_3-x_1)^2$ for  some constants $a,b,c$.

Equation (H II)  is in fact the integrability condition of the Lax pair 
 \begin{equation}\label{eqn3}
{\frac{\d Y}{\d x}}  =  \left( 
{\frac{\mu I}{(x-x_1)(x-x_2)(x-x_3)}}+\sum_{i=1}^3{\frac{\lambda_iC}{x-x_i}} \right)Y 
\end{equation}
\begin{equation} \label{eqn4}
{\frac{\d Y}{\d t}} =  \left(\nu I+\sum_{i=1}^3 \lambda_ix_iC\right)Y-Q(x) 
{\frac{\partial Y}{\partial x} } 
\end{equation}
where
\[ Q(x)=x^2+a(x_1-x_2)^2+b(x_2-x_3)^2+c(x_3-x_1)^2\]
and where $x_i=x_i(t)$ are parameterized (simple) singularities, $C$ is a constant traceless $2\times 2$ matrix, $I$  is the identity  matrix, 
$\mu\ne 0$ and $\lambda_i$   are constants such that  $\lambda_1+\lambda_2+\lambda_3=0$ (there is hence no singular point at $\infty$), and the function $\nu$ is a solution of 

\begin{eqnarray*}\label{eqn5}
{\frac{\d \nu}{\d x}}=-\mu {\frac{x+x_1+x_2+x_3}{(x-x_1)(x-x_2)(x-x_3)}}.
\end{eqnarray*}

Note that since the solutions of the latter equation are not rational in $x$, 
Equation~(\ref{eqn3})  is  not isomonodromic, by Sibuya's criterion. To describe the monodromy of this equation, let us fix  a fundamental solution $Y$ of the Lax pair at some $x_0$ not belonging to fixed disks $D_i$ with centers  $x_i(t)$, for all $i$.  Note that $Y$ must be analytic in both $x$ and $t$. A computation shows that the parameterized monodromy  matrix of Equation~(\ref{eqn3}) with respect to $Y$ and $x_i(t)$ is 
\[M_i(t)=e^{-2\pi \sqrt{-1} \mu \int_{t_0}^t\beta_i(t)dt}\ e^{2\pi \sqrt{-1} L_i(t_0)}\]
where $L_i(t)$ is an analytic function of $t$ such that, for some fundamental solution $Y_0$ of  Equation~(\ref{eqn3}) in the neighbourhood of given non-singular point $x_0$, the analytic extension of $Y_0$  to a neighbourhood of $x_i(t)$ is
\[{Y(t,x)=Y_i(t,x-x_i(t)).(x-x_i(t))^{L_i(t)}}\]
 where $Y_i$ is single-valued.  The coefficients $\beta_i(t)$ are given by 
\[\frac{x+\sum_{i=1}^3x_i}{\prod_{i=1}^3(x-x_i(t))}=\sum_{i=1}^3 \frac{\beta_i(t)}{x-x_i(t)}.\]
The  monodromy matrix  is for each $i$   of the form
\[M_i(t)=c_i(t)\ M_i(t_0)\]
with 
\[ \quad  c_i(t)=e^{-2\pi \sqrt{-1} \mu \int_{t_0}^t\beta_i(t)dt},\quad M_i(t_0)=e^{2\pi \sqrt{-1} L_i(t_0)},\] 
that is,  Equation~(\ref{eqn3}) is projectively isomonodromic. Moreover it is an example of a Fuchsian projectively isomonodromic equation to  which Proposition \ref{Fprop} applies, since we can write this equation

 \begin{eqnarray*} \frac{\partial Y}{\partial x} = \left(\sum_{i=1}^3 \frac{A_i(t)}{(x-x_i)}\right) Y
\end{eqnarray*}
where
\begin{eqnarray*}
A_i(t) & = & B_i(t) + b_i(t)I
\end{eqnarray*}
\[B_i(t)  =  \lambda_i C, \quad
b_i(t) = \frac{\mu}{\prod_{j\neq i}(x_i-x_j)}\]
and where
 \begin{eqnarray*} \frac{\partial Y}{\partial x} = \left(\sum_{i=1}^3 \frac{\lambda_i C}{(x-x_i)}\right) Y
\end{eqnarray*} 
 is clearly isomonodromic.

\section{Inverse problems}
\subsection{A parameterized version of the weak Riemann-Hilbert problem}
In \cite{MS0} we adapted Bolibrukh's techniques and  construction of holomorpic bundles (\cf \cite{anosov_bolibruch}, \cite{Bo1}, \cite{Bo2}, \cite{Bo3}, \cite{Bo3b}, \cite{BoMaMi}) to give a parameterized version of the weak Riemann-Hilbert problem. 

\begin{theorem}\label{weakRH} Let $S=\{a_1, \ldots a_s\}$ be a finite subset of $\PX^1(\CX)$ and $D$ an open  polydisk  in $\CX^r$. Let $\gamma_1, \ldots , \gamma_s$ be generators of $\pi_1(\PX^1(\CX)\backslash S ; a_0)$ for some fixed base-point $a_0\in\PX^1(\CX)\backslash S$,  and  let $M_i:D\rightarrow \GL_n(\CX)$,   $i =1 , \ldots , s$, be analytic maps with $M_1\cdot\ldots\cdot M_s=I_n$. There exists a parameterized linear differential system 
 \begin{eqnarray*}\label{eqnRH2}
\d_x Y &= &A(x,t)Y
 \end{eqnarray*}
with $A\in\gl_n(\calO_{D'}(x))$ for some open polydisk $D' \subset D$,  with only regular singular points, all in $S$,  such that for some parameterized fundamental solution, the parameterized monodromy matrix along each $\gamma_i$ is  $M_i$. Furthermore, given any $a_i \in \{a_1, \ldots, a_s\}$, the entries of $A$ may be chosen to have at worst  simple poles at all $a_j \neq a_i$.
\end{theorem}
The proof, as in the non-parameterized case,  here relies  on a parameterized version of the Birkhoff-Grothendieck theorem 
({\it cf.} \cite{ma80a}, Proposition 4.1;   \cite{bolibruch02}, Theorem 2;  \cite{bolibruch04}, Theorem A.1).

\subsection{ The inverse problem of PPV-theory} In analogy again with the non-pa\-ra\-me\-te\-ri\-zed case, we deduce in \cite{MS0}  the following consequence of the parameterized versions Theorem \ref{schlesinger} of Schlesinger's theorem and Theorem \ref{weakRH} above of the weak Riemann-Hilbert problem. As before, let $t=(t_1,\ldots,t_r)$ be a multi-parameter and   $\Delta_t
=\{\d_{t_1},\ldots,\d_{t_r}\}$ the corresponding partial derivations. We consider the differential  field $k=k_0(x)$, where $k_0$ is a $\Delta_t$-differentially closed field containing $\CX(t_1,\ldots,t_r)$, and $k$ is endowed with the derivations $\Delta=\{\d_x, \d_{t_1},\ldots,\d_{t_r}\}$. 

\begin{theorem}\label{inverseprop} Let $G$ be a $\Delta_t$-linear differential algebraic group defined over $k_0$ and assume that $G(k_0)$ contains a finitely generated Kolchin-dense subgroup $H$. Then $G(k_0)$ is the PPV-group of a PPV-extension of $k=k_0(x)$. \end{theorem}

The  condition in Theorem \ref{inverseprop},  that $G(k_0)$ contains a finitely generated Kolchin-dense subgroup $H$, characterizes indeed those linear differential algebraic groups over $k_0$ which are PPV-groups. The fact that the condition is  also necessary  was proved by  Dreyfus  \cite{TD} as a consequence of his parameterized version of Ramis's density theorem (see for example \cite{PuSi2003} p. 238). Ramis's theorem says that the (local) differential Galois group  over $\CX(\{x\})$ (local  at $0$) of a linear differential system of order $n$ is the Zariski-closure in $\GL(n,\CX)$ of a subgroup finitely generated by the so-called formal monodromy,  Stokes matrices and  exponential torus, together also called {\em generalized monodromy data}, which generalize to irregular singularities the notion of monodromy matrices for  regular singularities. Moreover, it can be proved that the (global) differential Galois group over $\CX(x)$ of a linear differential system is  the Zariski-closure of the subroup generated by the finitely many ``local" differential Galois groups just mentionned,  which can be simultaneously embedded as subgroups in the global PV-group.  Dreyfus \cite{TD} defines a parameterized version of the generalized monodromy data and gives a parameterized version of this theorem, which in turn gives the converse result of Theorem \ref{inverseprop} above.

\smallskip 
In the non-parameterized case, the solution by Tretkoff and Tretkoff \cite{tretkoff79} of the  differential Galois inverse problem over $\CX(x)$ uses the fact, proved by the same authors, that over an algebraically closed field of characteristic zero, any linear algebraic group is  the Zariski closure of some finitely generated subgroup. The latter does not hold though for linear {\em differential} algebraic groups. This can in particular be seen on  the additive group $\Ga(k_0)$ (using  notation as above for the differential field $k_0$)  which has the striking property that the Kolchin-closure of any of its  finitely generated subgroups is  a proper subroup of  $\Ga(k_0)$ (\cf \cite{MS0}). In \cite{Landesman} and \cite{CaSi} it is furthermore  shown that neither $\Ga(k_0)$ nor  $\Gm(k_0)$ is  the PPV-group of any PPV-extension of $k_0(x)$. In  \cite{sin11},  Singer proves the following result, using Corollary~\ref{inverseprop}. 
\begin{theorem} With notation as above, a linear algebraic group $G$ defined over $k_0$ is a PPV-group of a PPV-extension of $k_0(x)$ if and only if the identity component of $G$ has no quotient isomorphic to $\Ga(k_0)$ or $\Gm(k_0)$.
\end{theorem} 

More recently, Minchenko, Ovchinnikov and Singer \cite{MOS} gave a characterization of  linear unipotent differential algebraic groups that can be realized as PPV-groups. 

\begin{theorem}[ Minchenko, Ovchinnikov, Singer] A unipotent linear differential algebraic group $G$ over $k_0$ is the Kolchin-closure of a finitely generated subgroup if and only if it has differential type $0$. 
\end{theorem}

The meaning here of ``differential type $0$" is that a so-called `differential dimension' be finite.  The latter is defined as the transcendence degree over $k_0$ of the `differential function field'  $k_0\langle G^0\rangle$ over $k_0$ of   the identity component $G^0$ of $G$. 
If $G\subset\GL(n,k_0)$, the differential function field of $G^0$, denoted $k_0\langle G^0\rangle$,  is  the quotient-field of $R/{\mathcal I}$, where  $R/{\mathcal I}$ is the differential coordinate ring of the group. More precisely,   $R/{\mathcal I}$ is  the  quotient of the ring of differential polynomials $k_0\{y_{1,1},\ldots y_{n,n}\}$ in $n^2$ differential indeterminates (differential with respect to $\Delta_t$) by the differential ideal $\mathcal I$ of those differential polynomials vanishing  on $G^0$. 

The same authors have also given a characterization in \cite{MOSS} of those reductive linear differential algebraic groups that can occur as PPV-groups over $k_0(x)$. In both \cite{MOS} and \cite{MOSS} the authors give algorithms to determine if the PPV-groups is of the relevant type and give algorithms to compute this group if it is.

\section{Appendix}
Let $(K,\d)$ be an ordinary differential field  and $K\{X\}$ the differential ring of differential polynomials in one differential variable. By definition $K\{X\}$ is the ring $K[X_0, X_1,\ldots,X_n,\ldots]$ of polynomials in  the indeterminates $X_0, X_1,\ldots,X_n,\ldots,$ with  the derivation $\d$   extended  by $\d X_i=X_{i+1}$ for all $i\ge 0$. In $K\{X\}$ one writes $X$ for $X_0$, $X'$ for $X_1$, and $X^{(i)}:=\d^{(i)}X$ for all $X_i$.
The {\em order} $o(f)$ of an element $f\in K\{X\}$  is defined as the least integer $n$ such that $f\in K[X_0, X_1,\ldots,X_n] $ if $f\notin K$,    and $o(f)=-1$ if $f\in K$.  For basic facts and model theoretic properties of the theory DCF of differential closed fields, we refer for instance to  \cite{DaMa}, \cite{DaMa2}, \cite{Ro63}.

The  following definition is close to the definition of  algebraic closedness. It is  due to~Blum\cite{blum68}, who simplified an earlier  definition    introduced by  Robinson \cite{Ro59}. 
\begin{definition} [Blum] \label{DCC}The differential field $(K,\d)$ is said to be {\em differentially closed} if for any $f,g\in K\{X\}$, $f\notin K$  with $o(g)< o(f)$, there is an $a\in K$ such that $f(a)=0$ and $g(a)\neq 0$. 
\end{definition}
This definition is for instance well illustrated on Example~\ref{toymodel} above
\begin{equation*}
\frac{dy}{dx}=\frac{t}{x} y.
\end{equation*}
Let us show that over $K(x)$, where $K$ is a differentially closed field containing $\CX(t)$, the obstruction  to Galois correspondence vanishes. We recall that the PPV-extension of this equation over $K(x)$ is $K(x,x^t,\log x)$ and that   an element $\sigma$ of the PPV-group is defined by 
\[\sigma(x^t)=a_{\sigma}x^t, \ \ \sigma(\log x)=\log x+c_{\sigma}\]
where  $a_{\sigma}\in K^*$ satisfies \[a_{\sigma}'' a_{\sigma} - a_{\sigma}'^2=0 \]
and
\[c_{\sigma}=\frac{a_{\sigma}'}{a_{\sigma}}, \]
and where $a_{\sigma}', a_{\sigma}''$ are  derivatives with respect to the derivation extending $d/dt$.

To avoid  that $\log x$ be invariant by the PPV-group (in which case the invariant field of the PPV-group would not be the base-field $K(x)$)  we need  at least  one $\sigma$  to be such that $\sigma(\log x)\neq \log x$, that is,   given by    $a_{\sigma}\in K^*$  such that 

\[a_{\sigma}'' a_{\sigma} - a_{\sigma}'^2=0,  \qquad  \frac{a_{\sigma}'}{a_{\sigma}}\neq 0.\]Ê

Since $K$ is differentially closed, such an element exists by Definition \ref{DCC} applied to  $f(X)=X'' X -X'^2$ and $g(X)=X'$. 

The definition of  general   (non-ordinary)  differentially closed fields is due  to Kolchin \cf \cite{kolchinDCF} who called them ``constrainedly closed". For ordinary differential fields, the definition below is  equivalent to Definition~\ref{DCC} above.
\begin{definition}[Kolchin] Let $K$ be a $\Delta$-differential field, endowed with a finite set $\Delta$ of commuting derivations on $K$. The field $K$ is $\Delta$-differentially closed if it has no proper  constrained extensions.
\end{definition}

The definition of constrained extensions is the following.

\begin{definition} Let $K$ be a $\Delta$-differential field. A differential extension $L$ of $K$ is said to be {\em constrained} if for any finite family of elements  $(\eta_1,\ldots,\eta_s)$ of $L$ there is a $\Delta$-differential polynomial $P\in K\{y_1,\ldots,y_s \}$  such that $P(\eta_1,\ldots,\eta_s)\neq 0$ whereas $P(\zeta_1,\ldots, \zeta_s)=0$ for any non-generic differential specialization $(\zeta_1,\ldots,\zeta_s)$ of $(\eta_1,\ldots,\eta_s)$ over $K$.
\end{definition}

In Kolchin's terminology, a differential specialization $\zeta=(\zeta_1,\ldots,\zeta_s)$ of $\eta=(\eta_1,\ldots,\eta_s)$ in some extension of $K$ is  {\em generic} if the defining ideals of $\zeta$ and $\eta$ in $K\{y_1,\ldots,y_s \}$ are the same. We refer to Kolchin's original work for  details  about these notions (\cf \cite{kolchinDCF}, \cite{DAAG}, \cite{kolchin_groups}). 
The differential closure  is defined in a similar way as the algebraic closure.  
\begin{definition}
Let $K$ be a $\Delta$-differential field. A~{\em differential closure}  of $K$ is a differential, differentially closed extension of 
$K$ which can be embedded in any  given differential, differentially closed extension of $K$. 
\end{definition} 

\begin{theorem}
A differential field $K$ has a unique differential closure. 
\end{theorem}
This result was proved  by Morley \cite{morley}, Blum \cite{blum68}, Shelah \cite{shelah73} and Kolchin \cite{kolchinDCF}.  
Unlike the algebraic closure though, the differential closure fails to be minimal, even in characteristic $0$. Although it had been conjectured by some authors to be minimal (\cf \cite{sacks72}),  Kolchin, Rosenlicht, and Shelah independently proved that it is not. Shelah \cite{shelah73}  in particular proved that the ordinary differential closure  $\tilde{\QX}$ of $\QX$ is not minimal by exhibiting an infinite, strictly decreasing sequence of differentially closed intermediate differential extensions of $\QX$ in $\tilde{\QX}$. 

\bigskip
\noindent
{\bf Acknowledgements.} I would like to thank Michael Singer for helpful comments and   Renat Gontsov for his thorough proofreading of the manuscript. I am also grateful to Phyllis Cassidy for her slides on Kolchin's approach to DCF.

\bibliographystyle{amsplain}

\begin{thebibliography}{10}
 
 
  \bibitem{anosov_bolibruch}
 D.~V.~Anosov and A.~A.~Bolibruch,
 \newblock The Riemann-Hilbert Problem,
 \newblock Vieweg, Braunschweig, Wiesbaden, 1994.
 
 
\bibitem{babbitt_varadarajan}
 D.~G.~Babbitt and V.~S.~Varadarajan,
 \newblock Deformations of nilpotent matrices over rings and reduction of analytic families of differential equations,
 \newblock Memoirs AMS 55 (325), 1985.


  \bibitem{blum68}
 L.~Blum,
 \newblock Generalized algebraic structures: A model theoretic approach. Ph.D. dissertation,
 \newblock Massachusetts Institute of Technology, Cambridge, MA, 1968.
  

\bibitem{Bo1} A.~A.~ Bolibruch,
\newblock The Rieman-Hilbert problem, 
\newblock Russian Math. Surveys, 45, 1-47, 1990. 

\bibitem{Bo2} A.~A.~ Bolibruch,
\newblock  On sufficient conditions for the positive solvability of the Riemann-Hilbert problem, 
\newblock Math. Notes. Acad. Sci. USSR, 51 (1),  110-117, 1992.

\bibitem{Bo3} A.~A.~ Bolibruch,
\newblock  On an analytic transformation to the standard Birkhoff  form, 
\newblock Proc. Steklov Inst. Math.   203 (3),  29-35, 1995.

\bibitem{Bo3b} A.~A.~ Bolibruch,
\newblock The 21st Hilbert Problem for Linear Fuchsian Systems, 
\newblock Proc. Steklov Inst. Math. 206 (5),  1-145, 1995.

 \bibitem{Bol_iso_def}
A.~A.~Bolibruch,
\newblock On Isomonodromic Deformations of Fuchsian Systems
\newblock  J.~Dynam.\ Contr.\ Sys., 3 (4),  589-604, 1997.

\bibitem{bolibruch02}
 A.~A.~Bolibruch,
\newblock {Inverse problems for linear differential equations with
              meromorphic coefficients},
\newblock  Isomonodromic deformations and applications in physics
              ({M}ontr\'eal, {QC}, 2000), CRM Proc. Lecture Notes 31, 3--25, 2002.
 
 \bibitem{bolibruch04}
 A.~A.~Bolibruch and A.~R.~Its and A.~A.~Kapaev,
 \newblock On the Riemann-Hilbert-Birkhoff inverse monodromy problem and the Painlev\'e equations. Algebra i Analiz, 16(1), 121-162, 2004.
 
\bibitem{BoMaMi}
A.~A.~Bolibruch and S.~Malek and C.~Mitschi,
\newblock On the generalized Riemann-Hilbert problem with irregular singularities,
\newblock Expositiones Mathematicae, 24(3),  235--272, 2006.

\bibitem{Borel2}
A.~Borel,
\newblock Essays in the History of Lie Groups and Algebraic Groups
\newblock American mathematical Society, 2001.


\bibitem{Buium}
A.~Buium
\newblock Differential Algebraic Groups of Finite Dimension
\newblock Springer Lecture Notes in Math.1506, 1992.



\bibitem{ElieCartan}
\'E.~Cartan, 
\newblock Les syst\`emes diff\'erentiels ext\'erieurs et leurs applications g\'eom\'etriques,
\newblock Expos\'es de g\'eom\'etrie XII, Hermann, Paris, 1945.

\bibitem{Cartan}
H.~Cartan, 
\newblock Th\'eorie \'el\'ementaire des fonctions analytiques d'une ou plusieurs variables complexes,
\newblock Hermann, Paris, 1961.

\bibitem{cassidy}
P.~J.~Cassidy,
\newblock Differential algebraic groups
\newblock { American Journal of Mathematics}, 94:891-954,~1972.

\bibitem{CaSi}
P.~J.~Cassidy, M.~F.~Singer,
\newblock Galois Theory of parameterized Differential Equations and Linear Differential Algebraic Groups,
\newblock { Differential Equations and Quantum Groups}, D. Bertrand et.~al., eds.,  IRMA Lectures in Mathematics and Theoretical Physics 9, 113-157,2006.

\bibitem{ChAb}
S.~Chakravarty, M.~J.~Ablowitz,
\newblock Integrability, monodromy evolving deformations, and self-dual Bianchi~IX systems,
\newblock { Physical Review Letters} 76(6), 857-860,1996.

\bibitem{GD1}
G.~Darboux,  
\newblock Syst\`emes orthogonaux,
\newblock Ann. Sc. \'E.N.S., 1e s\'erie, tome 3, 97 -141, 1866.

\bibitem{GD2}
G.~Darboux, 
\newblock M\'emoire sur la th\'eorie des coordonn\'ees curvilignes, et des syst\`emes orthogonaux,
\newblock Ann. Sc. \'E.N.S., 2e s\'erie, tome 7, 101-150, 227-260, 275-348, 1878.
 
 
\bibitem{TD}
T.~Dreyfus,
\newblock A parameterized density theorem in differential Galois
theory
\newblock Pacific J. Math. {271}(1) 87-141, 2014.

\bibitem{H1}
G.~H.~Halphen,
\newblock Sur un syst\`eme dÕ\'equations diff\'erentielles,
\newblock C. R. Acad. Sci. 92, 1101-1103, 1881.

\bibitem{H2}
G.~H.~Halphen,
\newblock Sur certains syst\`emes dÕ\'equations diff\'erentielles,
\newblock C. R. Acad. Sci. 92,  1404-1406, 1881.

\bibitem{kolchinDCF}
E.~Kolchin,
\newblock Constrained extensions of differential fields,
\newblock {\it Advances in Math.} 12 (2), 141-170, 1974.

\bibitem{DAAG}
E.~R.~Kolchin,
\newblock Differential Algebra and Algebraic Groups,
\newblock Academic Press, 1976.

\bibitem{kolchin_groups}
E.~R.~Kolchin,
\newblock Differential algebraic groups
\newblock Academic Press, New York, 1985

\bibitem{Kuga}
M.~Kuga,
\newblock Galois'Dream : Group theory and differential equations
\newblock BirkhŠuser, 1993.

\bibitem{Landesman}
P.~Landesman,
\newblock Generalized differential Galois theory,
\newblock {\it Trans. Amer. Math. Soc.} 360(8), 4441--4495, 2008.

\bibitem{ma80a}
B.~Malgrange,
\newblock Sur les d\'eformations isomonodromiques. {I}. {S}ingularit\'es r\'eguli\`eres,
 \newblock { Mathematics and physics (Paris, 1979/1982)}, Progr. Math. 37, 401-426,1983.

\bibitem{DaMa}
D.~Marker,
\newblock Model theory of differential fields
 \newblock {t Preprint, University of Illinois at Chicago}, 2000. (\cf http://www.msri.org/publications/books/Book39/files/dcf.pdf)
 
 \bibitem{DaMa2}
D.~Marker,
\newblock Model theory, algebra and geometry
 \newblock { Math. Sci. Res. Inst. Publ.} 39,  53-63,  Cambridge Univ. Press.
 
\bibitem{MS0}
C.~Mitschi, M.~F.~Singer,
\newblock Monodromy groups  of 
parameterized linear differential equations with regular singularities
\newblock Proc. of the Amer. Math. Soc. { 141},  605-617 (2011). 

\bibitem{MS}
C.~Mitschi, M.~F.~Singer,
\newblock Projective Isomonodromy and Galois Groups,
\newblock Bull. London Math. Soc. 44 (5),   913-930 (2012).

\bibitem{MOS}
A.~Minchenko, A.~Ovchinnikov, M.~F.~Singer,
\newblock Unipotent differential algebraic groups as parameterized differential Galois groups
\newblock Journal of the Institute of Mathematics of Jussieu {13} (04),  671-700 (2014)

\bibitem{MOSS}
A.~Minchenko, A.~Ovchinnikov, M.~F.~Singer,
\newblock Reductive linear differential algebraic groups and the Galois groups of parameterized linear differential equations\newblock International Mathematics Research Notices (2013), to appear.

\bibitem{morley}
M.~D.~Morley,
\newblock Categoricity in power,
\newblock Trans. Amer. Math. Soc. 114,  514-538, 1965.

\bibitem{ohyama}
Y.~Ohyama,
\newblock Quadratic equations and monodromy evolving deformations
\newblock {\it arXiv} :0709.4587v1 [math.CA]  28 Sep 2007.

\bibitem{ohyama2}
Y.~Ohyama,
\newblock Monodromy evolving deformations and Halphen's equation
\newblock in Groups and Symmetries, {\it CRM Proc. Lecture Notes} 47, Amer.Math.Soc., Providence, RI, 2009. 



\bibitem{Pa78}
R.~S.~ Palais,
\newblock Some analogues of Hartogs theorem in an algebraic setting,
\newblock {\em  Amer. J. Math.} 100(2), 387--405, 1978.


\bibitem{PuSi2003}
M.~van der Put and M.~F.~Singer,
\newblock Galois Theory of Linear Differential Equations,
\newblock {\it Grundlehren der mathematischen Wissenshaften} 328, Springer-Verlag, 2003.

\bibitem{Ro59}
A.~ Robinson,
\newblock On the concept of differentially closed field,
\newblock {\em  Bull. Res. Council  Israel Sect. F} 8,  113-118, 1959.

\bibitem{Ro63}
A.~ Robinson,
\newblock Introduction to Model Theory and the Metamathematics of Algebra,
\newblock {\em  North Holland Publ., Amsterdam, }  1963

\bibitem{sacks72}
G.~E.~Sacks,
\newblock The  differential closure of a differential field,
\newblock {\em  Bull. Amer.Math. Soc.} 78 (5),  629-634, 1972.

\bibitem{schaefke}
R.~Sch\"afke,
\newblock Formal fundamental solutions of irregular singular differential equations depending on parameters,
\newblock { J. Dynam. Control Systems } 7 (4), 501--533, 2001.

\bibitem{schlesinger}
L.~Schlesinger,
\newblock {Handbuch der Theorie der Linearen Differentialgleichungen},
\newblock Teubner, Leipzig, 1887.

\bibitem{sei58}
A.~Seidenberg,
\newblock Abstract differential algebra and the analytic case,
\newblock { Proc. Amer. Math. Soc.} 9,  159-164, 1958.

\bibitem{sei69}
A.~Seidenberg,
\newblock Abstract differential algebra and the analytic case {II}.
\newblock { Proc. Amer. Math. Soc.} 23,  689691, 1969.

\bibitem{shelah73}
S.~Shelah,
\newblock Differentially closed fields,
\newblock Israel. J. Math. 25,  314-328, 1976.

\bibitem{sin11}
M.~F.~Singer,
\newblock Linear algebraic groups as parameterized Picard-Vessiot Galois groups,
\newblock Preprint, 2011.

\bibitem{Sibuya}
Y.~Sibuya,
\newblock Linear Differential Equations in the Complex Domain: Problems of Analytic Continuation,
\newblock { Translations of Mathematical Monographs, Volume 82}
\newblock American Mathematical Society, 1990

\bibitem{tretkoff79}
C.~Tretkoff and M.~Tretkoff,
\newblock Solution of the Inverse Problem in Differential {Galois} Theory in the Classical Case,
\newblock Amer. J. Math. 101, 1327-1332, 1979.

\bibitem{wood}
C.~Wood,
\newblock The model theory of differential fields revisited,
\newblock Israel Journal of Mathematics 25,  1976.

\bibitem{Zola}
H.~Zoladek,
\newblock The Monodromy Group,
\newblock Monografie matematyszne, Inst. Mat. PAN,  67, New Series, Birkh\"auser, 2006.

\end{thebibliography}

\end{document}